\documentclass[11pt]{amsart}
\usepackage[a4paper,margin=29mm]{geometry}
\usepackage{amsmath,amssymb,amsthm,mathtools}
\usepackage{mathrsfs}
\usepackage{microtype}
\usepackage{enumitem}
\usepackage{hyperref}
\hypersetup{colorlinks=true,citecolor=blue,linkcolor=blue,urlcolor=blue}

\newtheorem{theorem}{Theorem}[section]
\newtheorem{proposition}[theorem]{Proposition}
\newtheorem{lemma}[theorem]{Lemma}

\theoremstyle{remark}
\newtheorem{remark}[theorem]{Remark}
\newcommand{\T}{\mathbb T}
\newcommand{\R}{\mathbb R}
\newcommand{\Z}{\mathbb Z}
\newcommand{\ii}{\mathrm i}
\newcommand{\eps}{\varepsilon}
\newcommand{\av}[1]{\langle #1\rangle}
\newcommand{\norm}[1]{\lVert #1\rVert}
\newcommand{\cD}{\mathcal D}

\newcommand{\cR}{\mathcal R}

\newcommand{\Id}{\mathrm{Id}}

\title[Reversible Duffing equations]{Lagrange Stability for Reversible Duffing Equations\\
with Quasi-Periodic Coefficients}
\author{Huining Xue}
\address{College of Artificial Intelligence, Lishui University, Lishui, Zhejiang, China}
\email {xhn\_lsxy627@lsu.edu.cn}

\begin{document}
\begin{abstract}
We consider
\[
 \ddot x+f(x,\omega t)\dot x+g(x,\omega t)=0,
\]
where
\[
 f(x,\theta)=\sum_{j=0}^{m}a_j(\theta)x^{2j+1},\qquad
 g(x,\theta)=x^{2n+1}+\sum_{j=0}^{n-1}b_j(\theta)x^{2j+1}.
\]
The coefficient functions are real analytic and even on the torus and the
frequency vector $\omega$ is Diophantine.  If $n\geq2(m+1)$, we construct
codimension-one reversible KAM tori accumulating at infinity and prove that
all solutions are bounded.  The main point is a finite normal-form procedure.
After the reversible polynomial reduction, a logarithmic Fourier cut-off is
introduced.  At the $v$-th step a truncated homological equation is solved on
a non-resonant action interval and the new error satisfies an explicit
finite-step recurrence.  Thus an arbitrarily small negative power of the large
action is reached after finitely many steps. Finally, the Largrangian stability and the existence of quasi-periodic solutions are proved by the reversible KAM theorem.
\end{abstract}

\thanks{This work was supported by the research initiation project of Lishui University(Grant No.QDZK112026001).}
\maketitle

\section{Introduction}
The Lagrange-stability problem asks whether every solution of a forced
superlinear oscillator exists for all time and remains in a bounded region
of phase space.  It originates in Littlewood's problems on oscillatory
differential equations \cite{Littlewood}.  A first decisive result was
obtained by Morris \cite{Morris}, who proved boundedness for a periodically
forced cubic Duffing equation.  Dieckerhoff and Zehnder
\cite{DieckerhoffZehnder} subsequently treated general polynomial
Hamiltonian Duffing equations.  Their argument reduces the dynamics at
large energy to a small perturbation of a twist map and applies Moser's
invariant-curve theorem \cite{Moser1962}.

For time-periodic superquadratic potentials, Levi
\cite{Levi1991} developed an action--angle normalization at infinity.
The genuinely quasi-periodic Hamiltonian problem is more delicate because
there is no single forcing period and hence no ordinary Poincar\'e map.
Levi and Zehnder \cite{LeviZehnder} proved boundedness for a broad class of
quasi-periodic potentials by a near-integrability argument in the extended
phase space.  In the polynomial Duffing setting with the coefficients of low regularity, invariant tori and
boundedness results were obtained by Yuan \cite{Yuan1998}.  More recently,
Xue and Yuan \cite{XueYuan} proved the existence of positive-measure KAM
tori clustering at infinity for Duffing equations with analytic
quasi-periodic coefficients.  A central device in their proof is the
decomposition into Fourier modes below and above
\[
 K=c_*\log A,
\]
where $A$ is the large-amplitude parameter.  Low modes are normalized on
large non-resonant action intervals, whereas analyticity makes the high
modes smaller than an arbitrarily prescribed power of $A$.

The equation considered here is not Hamiltonian.  For reversible
time-periodic equations, Liu applied reversible KAM theory to boundedness
questions \cite{Liu1991}; Liu and Zanolin \cite{LiuZanolin} then established
the sharp polynomial result for
\[
 \ddot x+f(x,t)\dot x+g(x,t)=0,
\qquad n\geq2(m+1).
\]
Related invariant-curve theorems for quasi-periodic reversible mappings
were developed by Liu \cite{Liu2005}; see also the monograph of Sevryuk
\cite{Sevryuk} and the finite-smoothness result of Li, Qi and Yuan
\cite{LiQiYuan}.  These results supply the final persistence mechanism, but
they do not by themselves reduce the large, quasi-periodically dependent
perturbation arising at infinity to the required small normal form.

We study the reversible Duffing equation
\begin{equation}\label{eq:main}
 \ddot x+f(x,\omega t)\dot x+g(x,\omega t)=0,
\end{equation}
where
\begin{equation}\label{eq:fg}
 f(x,\theta)=\sum_{j=0}^{m}a_j(\theta)x^{2j+1},\qquad
 g(x,\theta)=x^{2n+1}+\sum_{j=0}^{n-1}b_j(\theta)x^{2j+1}.
\end{equation}
To the best of our knowledge, the combination of a non-Hamiltonian
reversible damping term, genuinely quasi-periodic coefficients, and the
sharp degree range has not been covered by the preceding results.

The contributions of the paper are the following.
\begin{enumerate}[label=\textup{(\roman*)}]
\item We combine the polynomial reversible reduction of
Liu--Zanolin with the logarithmic Fourier cut-off of Xue--Yuan.  The mixed
divisors
\[
 D_{k\ell}(\rho)=k c_0\rho^{n/(n+2)}
                 +\langle\ell,\omega\rangle
\]
are controlled on explicitly constructed non-resonant action intervals.
\item The large perturbation is reduced by a finite, rather than
convergent, normalization.  At the $v$-th step the new remainder obeys the
quantitative recurrence
\[
 \eps_{v+1}\leq C_vA^{-\kappa_0}\eps_v+
 B_v\eps_v^2+CK^dA^{-c_*\sigma_v}\eps_v.
\]
This gives $\eps_M\leq A^{-L}$ after an explicitly chosen finite number
$M$ of steps.
\item Every coordinate change is shown to commute with the reversing
involution.  Thus reversibility is preserved at each finite step, not only
for the final normal form.
\item The final invariant $(d+1)$-tori are codimension-one graphs in the
extended phase space.  They separate the extended cylinder and therefore
act as barriers, giving boundedness without introducing a Poincar\'e map.
\end{enumerate}

Let $\T^d=\R^d/(2\pi\Z)^d$.  If $s>0$, then
\[
 \T_s^d=\{\theta\in\mathbb C^d/(2\pi\Z)^d:
               |\operatorname{Im}\theta|<s\}.
\]
Our assumptions are:
\begin{enumerate}[label=\textup{(H\arabic*)}]
\item\label{H1}
$a_j,b_j$ extend holomorphically to $\T_{s_0}^d$, are real on $\T^d$, and
\[
 a_j(-\theta)=a_j(\theta),\qquad b_j(-\theta)=b_j(\theta);
\]
\item\label{H2}
for some $\gamma\in(0,1]$,
\begin{equation}\label{eq:DC}
 |\langle k,\omega\rangle|
 \geq\frac{\gamma}{|k|^{d+2}},
 \qquad 0\neq k\in\Z^d;
\end{equation}
\item\label{H3}
\begin{equation}\label{eq:degree}
 n\geq2(m+1).
\end{equation}
\end{enumerate}

\begin{theorem}\label{thm:main}
Under \ref{H1}--\ref{H3}, every solution of \eqref{eq:main} is defined on
$\R$ and
\[
 \sup_{t\in\R}\bigl(|x(t)|+|\dot x(t)|\bigr)<\infty.
\]
Moreover, the extended system possesses a positive-measure family of
analytic invariant $(d+1)$-tori accumulating at infinity.  The frequency
vector on each such torus is $(\omega,\Omega)$, where
$|\Omega|\to\infty$ along the family.
\end{theorem}

\begin{remark}
Condition \eqref{eq:degree}, rather than $n\geq m+1$, is the sharp degree
condition in the reversible polynomial class \eqref{eq:fg}; see
\cite{LiuZanolin}.  The evenness in \ref{H1} is also an actual assumption:
it cannot be imposed without loss of generality.
\end{remark}

\section{Action--angle coordinates}

Put $y=\dot x$ and $\theta=\omega t$.  Then \eqref{eq:main} is equivalent to
\begin{equation}\label{eq:extended}
 \dot x=y,\qquad
 \dot y=-f(x,\theta)y-g(x,\theta),\qquad
 \dot\theta=\omega.
\end{equation}
If
\[
 \cR(x,y,\theta)=(-x,y,-\theta),
\]
then
\begin{equation}\label{eq:rev}
 X(\cR z)=-D\cR(z)X(z).
\end{equation}

Let $(S(t),C(t))$ be the solution of
\begin{equation}\label{eq:SC}
 \dot S=C,\qquad \dot C=-S^{2n+1},\qquad (S(0),C(0))=(0,1),
\end{equation}
and let $T_0$ be its least period.  We have
\begin{equation}\label{eq:SCproperties}
 S(-t)=-S(t),\quad C(-t)=C(t),\quad
 (n+1)C(t)^2+S(t)^{2n+2}=n+1.
\end{equation}
Set
\begin{equation}\label{eq:ab}
 \alpha=\frac1{n+2},\qquad
 \beta=\frac{n+1}{n+2},\qquad q=2\beta-1=\frac n{n+2}.
\end{equation}
Choose $c_1,c_2>0$ so that
\begin{equation}\label{eq:AA}
 x=c_1\rho^\alpha S(T_0\varphi),\qquad
 y=c_2\rho^\beta C(T_0\varphi)
\end{equation}
has Jacobian one.  This is an analytic area-preserving diffeomorphism
$\R_+\times\T\to\R^2\setminus\{0\}$.

For $S=S(T_0\varphi)$ and $C=C(T_0\varphi)$, direct substitution gives
\begin{equation}\label{eq:AAfield}
 \dot\rho=F_1(\rho,\varphi,\theta),\qquad
 \dot\varphi=\nu(\rho)+F_2(\rho,\varphi,\theta),\qquad
 \dot\theta=\omega,
\end{equation}
where
\begin{equation}\label{eq:nu}
 \nu(\rho)=c_0\rho^q,\qquad c_0>0,
\end{equation}
and, for suitable positive constants $\kappa_\ell$,
\begin{align}
F_1={}&-\kappa_1\sum_{j=0}^{m}
 a_j(\theta)\rho^{1+(2j+1)\alpha}S^{2j+1}C^2
 \kappa_2\sum_{j=0}^{n-1}
 b_j(\theta)\rho^{(2j+2)\alpha}S^{2j+1}C,
 \label{eq:F1}\\
F_2={}&\kappa_3\sum_{j=0}^{m}
 a_j(\theta)\rho^{(2j+1)\alpha}S^{2j+2}C
 \kappa_4\sum_{j=0}^{n-1}
 b_j(\theta)\rho^{(2j+2)\alpha-1}S^{2j+2}.
 \label{eq:F2}
\end{align}
Changing the sign of $\kappa_2$ or $\kappa_4$ according to the normalization
in \eqref{eq:AA} does not affect any estimate below.

The involution induced by \eqref{eq:AA} is
\[
 \cR_0(\rho,\varphi,\theta)=(\rho,-\varphi,-\theta),
\]
and \eqref{eq:rev} is equivalent to
\begin{equation}\label{eq:Fparity}
 F_1(\rho,-\varphi,-\theta)=-F_1(\rho,\varphi,\theta),\qquad
 F_2(\rho,-\varphi,-\theta)= F_2(\rho,\varphi,\theta).
\end{equation}
Also,
\begin{equation}\label{eq:zeroaverage}
 \int_0^{2\pi}S(T_0\varphi)^{2j+1}
 C(T_0\varphi)^2\,d\varphi=0.
\end{equation}

\begin{lemma}[Equivariant changes preserve reversibility]
\label{lem:revchange}
Let $z=(r,\psi)$, $S=\operatorname{diag}(1,-1)$, and
\[
 \mathscr G(z,\theta)=(Sz,-\theta).
\]
Suppose the extended vector field
$\widetilde X=(X(z,\theta),\omega)$ satisfies
\begin{equation}\label{eq:revabstract}
 \widetilde X(\mathscr G\xi)
 =-D\mathscr G\,\widetilde X(\xi).
\end{equation}
Let
\[
 \widetilde\Phi(z,\theta)=(\Phi(z,\theta),\theta)
\]
be an analytic diffeomorphism satisfying
\begin{equation}\label{eq:equivariance}
 \widetilde\Phi\circ\mathscr G
 =\mathscr G\circ\widetilde\Phi .
\end{equation}
Then the pull-back vector field
\begin{equation}\label{eq:pullback}
 \widetilde Y(\xi)
 =D\widetilde\Phi(\xi)^{-1}
 \widetilde X(\widetilde\Phi(\xi))
\end{equation}
is reversible with respect to $\mathscr G$.

If $\Phi(z,\theta)=z+W(z,\theta)$ and $W=(U,V)$, condition
\eqref{eq:equivariance} is equivalent to
\begin{equation}\label{eq:UVequiv}
 U(r,-\psi,-\theta)=U(r,\psi,\theta),\qquad
 V(r,-\psi,-\theta)=-V(r,\psi,\theta).
\end{equation}
\end{lemma}

\begin{proof}
Differentiating \eqref{eq:equivariance} gives
\begin{equation}\label{eq:diffEquiv}
 D\widetilde\Phi(\mathscr G\xi)D\mathscr G
 =D\mathscr G\,D\widetilde\Phi(\xi).
\end{equation}
Using \eqref{eq:revabstract}, \eqref{eq:equivariance}, and
\eqref{eq:diffEquiv}, we find
\begin{align*}
 \widetilde Y(\mathscr G\xi)
 &=
 D\widetilde\Phi(\mathscr G\xi)^{-1}
 \widetilde X(\widetilde\Phi(\mathscr G\xi))\\
 &=
 D\widetilde\Phi(\mathscr G\xi)^{-1}
 \widetilde X(\mathscr G\widetilde\Phi(\xi))\\
 &=
 -D\widetilde\Phi(\mathscr G\xi)^{-1}
 D\mathscr G\,\widetilde X(\widetilde\Phi(\xi))\\
 &=-D\mathscr G\,D\widetilde\Phi(\xi)^{-1}
 \widetilde X(\widetilde\Phi(\xi))
 =-D\mathscr G\,\widetilde Y(\xi).
\end{align*}
Finally, $\Phi(Sz,-\theta)=S\Phi(z,\theta)$ is equivalent to
$W(Sz,-\theta)=SW(z,\theta)$, which is precisely
\eqref{eq:UVequiv}.
\end{proof}

\begin{lemma}\label{lem:symbol}
For $p\geq0$ and every multi-index $\ell$ in $(\varphi,\theta)$,
\begin{align}
 |\partial_\rho^p\partial_{(\varphi,\theta)}^\ell F_1|
 &\leq C_{p,\ell}
 \left(\rho^{1+(2m+1)\alpha-p}+\rho^{2n\alpha-p}\right),
 \label{eq:F1bound}\\
 |\partial_\rho^p\partial_{(\varphi,\theta)}^\ell F_2|
 &\leq C_{p,\ell}
 \left(\rho^{(2m+1)\alpha-p}+
       \rho^{2n\alpha-1-p}\right).
 \label{eq:F2bound}
\end{align}
\end{lemma}

\begin{proof}
For the $j$-th damping term in \eqref{eq:F1} and \eqref{eq:F2},
\begin{align*}
 \partial_\rho^p
 \rho^{1+(2j+1)\alpha}
 &=\prod_{\ell=0}^{p-1}
 \bigl(1+(2j+1)\alpha-\ell\bigr)
 \rho^{1+(2j+1)\alpha-p},\\
 \partial_\rho^p
 \rho^{(2j+1)\alpha}
 &=\prod_{\ell=0}^{p-1}
 \bigl((2j+1)\alpha-\ell\bigr)
 \rho^{(2j+1)\alpha-p}.
\end{align*}
The corresponding restoring terms have powers
$(2j+2)\alpha-p$ and $(2j+2)\alpha-1-p$.
Taking $j=m$ and $j=n-1$, respectively, proves
\eqref{eq:F1bound}--\eqref{eq:F2bound}, because $S,C$ and all their
derivatives are bounded.
\end{proof}

\section{Finite normalization of the large perturbation}

\subsection{Complex domains, norms and Fourier cut-off}

Fix a large $A$ and a closed interval
$I\subset[A,2A]$.  For $r,s>0$, write
\[
 \cD(I,r,s)=\{(\rho,\varphi,\theta):
 \operatorname{dist}(\rho,I)<r,\
 |\operatorname{Im}\varphi|<s,\
 |\operatorname{Im}\theta|<s\}.
\]
For a scalar function $h$ define
\[
 |h|_{r,s}=\sup_{\cD(I,r,s)}|h|,
\]
and for a vector $H=(H_1,H_2)$ use the weighted norm
\begin{equation}\label{eq:weightednorm}
 \norm{H}_{r,s}
 =A^{-1}|H_1|_{r,s}+|H_2|_{r,s}
  +|\partial_\rho H_1|_{r,s}
  +A|\partial_\rho H_2|_{r,s}.
\end{equation}

If
\[
 h(\rho,\varphi,\theta)
 =\sum_{(k,\ell)\in\Z^{1+d}}
 \widehat h(\rho,k,\ell)
 e^{\ii(k\varphi+\langle\ell,\theta\rangle)},
\]
let
\begin{equation}\label{eq:PiK}
 \Pi_Kh=\sum_{|k|+|\ell|\leq K}
 \widehat h(\rho,k,\ell)
 e^{\ii(k\varphi+\langle\ell,\theta\rangle)},\qquad
 \Pi_K^\perp h=h-\Pi_Kh .
\end{equation}
For $0<\sigma<s$,
\begin{equation}\label{eq:Fouriertail}
 |\Pi_K^\perp h|_{r,s-\sigma}
 \leq C_d K^d e^{-K\sigma}|h|_{r,s}.
\end{equation}
Indeed,
\[
 |\widehat h(\rho,k,\ell)|
 \leq |h|_{r,s}e^{-(|k|+|\ell|)s},
\]
and
\[
 \sum_{N>K}\#\{(k,\ell):|k|+|\ell|=N\}e^{-N\sigma}
 \leq C_dK^de^{-K\sigma}.
\]
We take
\begin{equation}\label{eq:K}
 K=c_*\log A.
\end{equation}

\subsection{The non-resonant interval}

Let
\begin{equation}\label{eq:divisor}
 D_{k\ell}(\rho)
 =k\nu(\rho)+\langle\ell,\omega\rangle.
\end{equation}
For $k\neq0$, define
\begin{equation}\label{eq:badset}
 \mathcal B_{k\ell}
 =\left\{\rho\in[A,2A]:
 |D_{k\ell}(\rho)|
 <\frac{\delta A^q}{(1+|\ell|)^{d+2}}\right\},
 \qquad \delta=K^{-4d-6}.
\end{equation}

\begin{lemma}\label{lem:nonresonance}
For all sufficiently large $A$, there is a closed interval
$I_A\subset[A,2A]$ such that
\begin{equation}\label{eq:IAlength}
 |I_A|\geq cAK^{-5d-8}
\end{equation}
and, for $\rho\in I_A$ and
$0<|k|+|\ell|\leq K$,
\begin{equation}\label{eq:NR}
 |D_{k\ell}(\rho)|\geq
 \begin{cases}
 \gamma(1+|\ell|)^{-d-2},& k=0,\\[2mm]
 \delta A^q(1+|\ell|)^{-d-2},& k\neq0.
 \end{cases}
\end{equation}
\end{lemma}

\begin{proof}
Since
\[
 \nu'(\rho)=c_0q\rho^{q-1},\qquad
 cA^{q-1}\leq|\nu'(\rho)|\leq CA^{q-1}
 \quad(A\leq\rho\leq2A),
\]
the mean value theorem gives, for $k\neq0$,
\begin{equation}\label{eq:badmeasure1}
 \operatorname{meas}\mathcal B_{k\ell}
 \leq
 \frac{C\delta A^q}
 {|k|A^{q-1}(1+|\ell|)^{d+2}}
 =
 \frac{C\delta A}
 {|k|(1+|\ell|)^{d+2}}.
\end{equation}
Therefore
\begin{align}
 \operatorname{meas}
 \bigcup_{\substack{0<|k|+|\ell|\leq K\\ k\neq0}}
 \mathcal B_{k\ell}
 &\leq
 C\delta A
 \sum_{1\leq|k|\leq K}\frac1{|k|}
 \sum_{\ell\in\Z^d}\frac1{(1+|\ell|)^{d+2}}
 \notag\\
 &\leq C\delta A\log K
 \leq \frac A4. \label{eq:badmeasure2}
\end{align}
The complement has at most $CK^{d+1}$ components.  One component
$I_A$ consequently satisfies
\[
 |I_A|\geq
 \frac{A/2}{CK^{d+1}}
 \geq cAK^{-5d-8}.
\]
This proves the second line of \eqref{eq:NR}; the first line is
\eqref{eq:DC}.
\end{proof}

Let
\begin{equation}\label{eq:r0}
 r_0=cAK^{-6d-10}.
\end{equation}
By Cauchy's estimate and \eqref{eq:divisor},
\[
 |D_{k\ell}(\rho)-D_{k\ell}(\rho_0)|
 \leq C|k|A^{q-1}r_0
 \leq \frac12\delta A^q(1+|\ell|)^{-d-2}
\]
for $\rho_0\in I_A$, $|\rho-\rho_0|<r_0$, and
$|k|+|\ell|\leq K$.  Thus \eqref{eq:NR} remains valid on the complex
$r_0$-neighbourhood of $I_A$, with its right-hand side divided by two.

\subsection{The finite induction}

The reversible polynomial reduction of Liu--Zanolin is used only once,
before the Fourier iteration.  In the analytic setting it yields, on
$\cD(I_A,r_0,s_0/4)$, a reversible transformation $\Psi_0$ such that
\begin{equation}\label{eq:initialNF}
 \begin{aligned}
 \dot\rho&=a_0(\rho,\theta)+F_0(\rho,\varphi,\theta),\\
 \dot\varphi&=\nu(\rho)+b_0(\rho,\theta)
                   +G_0(\rho,\varphi,\theta),\\
 \dot\theta&=\omega,
 \end{aligned}
\end{equation}
For clarity, the reversibility of $\Psi_0$ can be verified at every
elementary polynomial reduction.  The radial and angular changes have,
respectively, the form
\begin{equation}\label{eq:elementaryChanges}
 \Phi^{(r)}(r,\psi,\theta)
 =(r+U(r,\psi,\theta),\psi,\theta),\qquad
 \Phi^{(\psi)}(r,\psi,\theta)
 =(r,\psi+V(r,\psi,\theta),\theta).
\end{equation}
If $h_1(r,-\psi,-\theta)=-h_1(r,\psi,\theta)$ and
$\av{h_1}_{\psi}=0$, the normalized primitive
\begin{equation}\label{eq:Uprimitive}
 U(r,\psi,\theta)
 =\frac1{\nu(r)}
 \left[
 \int_0^\psi h_1(r,\tau,\theta)\,d\tau
 -\av{\int_0^\psi h_1(r,\tau,\theta)\,d\tau}_{\psi}
 \right]
\end{equation}
satisfies
\[
 U(r,-\psi,-\theta)=U(r,\psi,\theta).
\]
Similarly, if
$h_2(r,-\psi,-\theta)=h_2(r,\psi,\theta)$ and
$\av{h_2}_{\psi}=0$, then
\begin{equation}\label{eq:Vprimitive}
 V(r,\psi,\theta)
 =\frac1{\nu(r)}
 \left[
 \int_0^\psi h_2(r,\tau,\theta)\,d\tau
 -\av{\int_0^\psi h_2(r,\tau,\theta)\,d\tau}_{\psi}
 \right]
\end{equation}
satisfies
\[
 V(r,-\psi,-\theta)=-V(r,\psi,\theta).
\]
Thus both maps in \eqref{eq:elementaryChanges} satisfy
\eqref{eq:equivariance}; Lemma \ref{lem:revchange} proves inductively that
every intermediate polynomial normal form, and hence \eqref{eq:initialNF},
is reversible.  Its coefficients satisfy
\begin{equation}\label{eq:initialparity}
 \begin{aligned}
 a_0(\rho,-\theta)&=-a_0(\rho,\theta),&
 b_0(\rho,-\theta)&= b_0(\rho,\theta),\\
 F_0(\rho,-\varphi,-\theta)&=-F_0(\rho,\varphi,\theta),&
 G_0(\rho,-\varphi,-\theta)&= G_0(\rho,\varphi,\theta),
 \end{aligned}
\end{equation}
and
\begin{equation}\label{eq:epsilon0}
 \av{F_0}_{\varphi}=\av{G_0}_{\varphi}=0,\qquad
 \norm{(F_0,G_0)}_{r_0,s_0/4}\leq
 \eps_0=A^{-\sigma_0}.
\end{equation}
If $\overline N_0=(a_0,b_0)^T$, the same reduction gives
$\kappa_0>0$ such that
\begin{equation}\label{eq:averageTame}
 \Gamma_A\left(
 |D_z\overline N_0|+
 r_0|D_z^2\overline N_0|
 \right)\leq A^{-\kappa_0},
\end{equation}
where $\Gamma_A$ is defined in \eqref{eq:Gamma}.  Thus the averaged part
has strictly lower symbolic order than the fast field $(0,\nu(\rho))^T$.
Here $\sigma_0>0$ depends only on $n,m$.

For reference, the power count behind \eqref{eq:epsilon0} is
\begin{equation}\label{eq:powercount}
 n-2(m+1)\geq0,\qquad
 \rho^{-q}\rho^{(2m+1)\alpha}
 =\rho^{-(n-2m-1)/(n+2)}.
\end{equation}
The leading radial damping monomial is first integrated in $\varphi$ using
\eqref{eq:zeroaverage}; substitution of this primitive in the angular
equation produces one additional symbolic derivative $\rho^{-1}$.  Thus
one complete radial--angular reduction has gain
\begin{equation}\label{eq:sigmagain}
 \rho^{-\sigma_*},\qquad
 \sigma_*=\frac1{n+2}
 \min\{1,n-2m-1\}>0.
\end{equation}
After a finite number of such polynomial reductions,
$\sigma_0>0$ in \eqref{eq:epsilon0}.  This is the only place where
\eqref{eq:degree} is used.

We now give the detailed finite Fourier normalization.

\begin{proposition}[Finite-step normal form]\label{prop:nf}
Let $L>0$.  If $c_*$ in \eqref{eq:K} and then $A$ are sufficiently large,
there exist an integer
\begin{equation}\label{eq:Mchoice}
 M=\min\{j\geq1:\sigma_0+\tfrac12j\kappa_0\geq L+4\}
\end{equation}
and reversible analytic transformations
\[
 \Phi_v=\Id+W_v,\qquad 0\leq v<M,
\]
such that
\[
 \Psi^{(M)}=\Psi_0\circ\Phi_0\circ\cdots\circ\Phi_{M-1}
\]
transforms \eqref{eq:AAfield} into
\begin{equation}\label{eq:NFfinal}
 \begin{aligned}
 \dot r&=a_M(r,\theta)+F_M(r,\psi,\theta),\\
 \dot\psi&=\nu(r)+b_M(r,\theta)+G_M(r,\psi,\theta),\\
 \dot\theta&=\omega,
 \end{aligned}
\end{equation}
where
\begin{equation}\label{eq:finalsmall}
 \norm{(F_M,G_M)}_{r_M,s_M}\leq A^{-L},
\end{equation}
\begin{equation}\label{eq:finalparity}
 \begin{aligned}
 a_M(r,-\theta)&=-a_M(r,\theta),&
 b_M(r,-\theta)&= b_M(r,\theta),\\
 F_M(r,-\psi,-\theta)&=-F_M(r,\psi,\theta),&
 G_M(r,-\psi,-\theta)&= G_M(r,\psi,\theta).
 \end{aligned}
\end{equation}
The final widths satisfy
\[
 r_M\geq\frac12r_0,\qquad s_M\geq\frac18s_0.
\]
\end{proposition}

\begin{proof}
We give the induction in five steps.

\smallskip
\noindent\emph{Step 1: induction hypotheses.}
Set
\begin{equation}\label{eq:widths}
 \sigma_v=\frac{s_0}{2^{v+6}},\qquad
 s_{v+1}=s_v-4\sigma_v,\qquad
 r_{v+1}=r_v-\frac{r_0}{2^{v+3}},
\end{equation}
with $s_0$ in this display replaced initially by the width $s_0/4$
of \eqref{eq:initialNF}.  Suppose that on
$\cD(I_A,r_v,s_v)$ the $v$-th system is
\begin{equation}\label{eq:vthsystem}
 \dot z=N_v(z,\theta)+R_v(z,\theta),\qquad
 \dot\theta=\omega,
\end{equation}
where $z=(r,\psi)$,
\begin{equation}\label{eq:NvRv}
 N_v=
 \binom{a_v(r,\theta)}
       {\nu(r)+b_v(r,\theta)},\qquad
 R_v=\binom{F_v(r,\psi,\theta)}
           {G_v(r,\psi,\theta)},
\end{equation}
\[
 \av{R_v}_{\psi}=0,\qquad
 \norm{R_v}_{r_v,s_v}\leq\eps_v,
\]
and, with $\overline N_v=(a_v,b_v)^T$,
\begin{equation}\label{eq:averageTameV}
 \Gamma_A\left(
 |D_z\overline N_v|+
 r_v|D_z^2\overline N_v|
 \right)\leq2A^{-\kappa_0}.
\end{equation}
and the parity relations in \eqref{eq:finalparity} hold with $M$ replaced by
$v$.  For $v=0$, these assertions are \eqref{eq:initialNF}--\eqref{eq:epsilon0}.

\smallskip
\noindent\emph{Step 2: the homological equation.}
Let
\[
 R_v^{\leq}=\Pi_KR_v,\qquad R_v^>=\Pi_K^\perp R_v.
\]
We solve the vector homological equation
\begin{equation}\label{eq:homological}
 \mathcal H_0W_v=R_v^{\leq},\qquad
 \mathcal H_0W
 =\omega\cdot\partial_\theta W+
 D_zW\,N_0-D_zN_0\,W,\qquad
 N_0=\binom0{\nu(r)},
\qquad \av{W_v}_{\psi}=0.
\end{equation}
Writing $W_v=(U_v,V_v)$ and
$R_v^{\leq}=(F_v^{\leq},G_v^{\leq})$, its solution is
\begin{align}
 U_v(r,\psi,\theta)
 &=\sum_{\substack{|k|+|\ell|\leq K\\k\neq0}}
 \frac{\widehat F_v(r,k,\ell)}
 {\ii D_{k\ell}(r)}
 e^{\ii(k\psi+\langle\ell,\theta\rangle)},\label{eq:Uv}\\
 V_v(r,\psi,\theta)
 &=\sum_{\substack{|k|+|\ell|\leq K\\k\neq0}}
 \left\{
 \frac{\widehat G_v(r,k,\ell)}
 {\ii D_{k\ell}(r)}
 +\frac{\nu'(r)\widehat F_v(r,k,\ell)}
 {(\ii D_{k\ell}(r))^2}
 \right\}
 e^{\ii(k\psi+\langle\ell,\theta\rangle)}.
\label{eq:Vv}
\end{align}
There are no $k=0$ terms because $\av{R_v}_{\psi}=0$.
By \eqref{eq:NR},
\begin{equation}\label{eq:Gamma}
 \Gamma_A=
 C\delta^{-1}A^{-q}K^{d+2}
 \leq CA^{-q}K^{5d+8}.
\end{equation}
Using Cauchy's estimate on the loss $\sigma_v$,
\begin{equation}\label{eq:Westimate}
 \norm{W_v}_{r_v-\frac{r_0}{2^{v+4}},s_v-\sigma_v}
 \leq
 C\Gamma_A\sigma_v^{-(d+2)}\eps_v
 =:\eta_v.
\end{equation}
Moreover, differentiating $D_{k\ell}^{-1}$ gives
\[
 \partial_r D_{k\ell}^{-1}
 =-\frac{k\nu'(r)}{D_{k\ell}(r)^2},
\]
which is included in the weighted norm in \eqref{eq:Westimate}.

The induction hypothesis gives, for real $r$,
\begin{equation}\label{eq:FourierParity}
 \widehat F_v(r,-k,-\ell)=-\widehat F_v(r,k,\ell),
 \qquad
 \widehat G_v(r,-k,-\ell)= \widehat G_v(r,k,\ell).
\end{equation}
Since
\begin{equation}\label{eq:divisorParity}
 D_{-k,-\ell}(r)=-D_{k\ell}(r),
\end{equation}
formulas \eqref{eq:Uv}--\eqref{eq:Vv} imply
\[
 \widehat U_v(r,-k,-\ell)=\widehat U_v(r,k,\ell),
 \qquad
 \widehat V_v(r,-k,-\ell)=-\widehat V_v(r,k,\ell).
\]
Equivalently,
\begin{equation}\label{eq:Wparity}
 U_v(r,-\psi,-\theta)=U_v(r,\psi,\theta),\qquad
 V_v(r,-\psi,-\theta)=-V_v(r,\psi,\theta).
\end{equation}
Therefore $\Phi_v=\Id+W_v$ commutes with
$(r,\psi,\theta)\mapsto(r,-\psi,-\theta)$.  Lemma
\ref{lem:revchange} now shows directly that the vector field after the
$v$-th change is reversible.

\smallskip
\noindent\emph{Step 3: exact transformed vector field.}
Use old coordinates $z=\Phi_v(z_+,\theta)=z_++W_v(z_+,\theta)$.
Then
\[
 \dot z=(I+D_zW_v)\dot z_+
       +\omega\cdot\partial_\theta W_v.
\]
Consequently,
\begin{equation}\label{eq:exactpush}
 X_{v+1}
 =(I+D_zW_v)^{-1}
 \left[
 N_v\circ\Phi_v+R_v\circ\Phi_v
 -\omega\cdot\partial_\theta W_v
 \right].
\end{equation}
Write $N_v=N_0+\overline N_v$.  Subtract
$(I+D_zW_v)N_v$ in \eqref{eq:exactpush} and use
\eqref{eq:homological}.  We obtain
\begin{equation}\label{eq:pushsplit}
 X_{v+1}
 =N_v+\mathcal E_v,
\end{equation}
where
\begin{equation}\label{eq:Ev}
 \mathcal E_v
 =(I+D_zW_v)^{-1}
 \left[
 R_v^>+\mathcal Q_v
 \right].
\end{equation}
Here
\begin{align}
 \mathcal Q_v={}&
 N_0\circ\Phi_v-N_0-D_zN_0W_v
 +\overline N_v\circ\Phi_v-\overline N_v
\notag\\
&-D_zW_v\,\overline N_v
 +R_v\circ\Phi_v-R_v .
\label{eq:Qv}
\end{align}
The Taylor remainders have the exact integral form
\begin{align}
 N_0\circ\Phi_v-N_0-D_zN_0W_v
 &=\int_0^1(1-t)
 D_z^2N_0(z+tW_v)[W_v,W_v]\,dt,
 \label{eq:TaylorN}\\
 R_v\circ\Phi_v-R_v
 &=\int_0^1
 D_zR_v(z+tW_v)W_v\,dt.
\label{eq:TaylorR}
\end{align}
Also,
\begin{equation}\label{eq:TaylorAverage}
 \overline N_v\circ\Phi_v-\overline N_v
 =\int_0^1D_z\overline N_v(z+tW_v)W_v\,dt.
\end{equation}

Define the new average and oscillatory parts by
\begin{equation}\label{eq:newaverage}
 N_{v+1}=N_v+\av{\mathcal E_v}_{\psi},\qquad
 R_{v+1}=\mathcal E_v-\av{\mathcal E_v}_{\psi}.
\end{equation}
Thus $\av{R_{v+1}}_{\psi}=0$ exactly.  More explicitly, reversibility gives
\[
 \mathcal E_{v,1}(r,-\psi,-\theta)
 =-\mathcal E_{v,1}(r,\psi,\theta),\qquad
 \mathcal E_{v,2}(r,-\psi,-\theta)
 = \mathcal E_{v,2}(r,\psi,\theta).
\]
Taking the $\psi$-average yields
\[
 \av{\mathcal E_{v,1}}_\psi(r,-\theta)
 =-\av{\mathcal E_{v,1}}_\psi(r,\theta),\qquad
 \av{\mathcal E_{v,2}}_\psi(r,-\theta)
 = \av{\mathcal E_{v,2}}_\psi(r,\theta).
\]
Consequently both $N_{v+1}$ and $R_{v+1}$ have the required parity, so
\eqref{eq:finalparity} propagates from $v$ to $v+1$.
Furthermore, \eqref{eq:newaverage}, Cauchy's estimate, and
\eqref{eq:recurrence0} below imply
\[
 \Gamma_A\left(
 |D_z\overline N_{v+1}|+
 r_{v+1}|D_z^2\overline N_{v+1}|
 \right)
 \leq
 \Gamma_A\left(
 |D_z\overline N_v|+
 r_v|D_z^2\overline N_v|
 \right)
 +C\Gamma_A\sigma_v^{-2}\eps_v .
\]
For large $A$, \eqref{eq:epsinduction} makes the last term at most
$2^{-v-2}A^{-\kappa_0}$; hence \eqref{eq:averageTameV} also propagates.

\smallskip
\noindent\emph{Step 4: the recurrence.}
Assume $\eta_v\leq c\min\{r_v-r_{v+1},\sigma_v\}$.
Then
\[
 \norm{(I+D_zW_v)^{-1}}\leq2.
\]
Equations \eqref{eq:Fouriertail},
\eqref{eq:Westimate}, \eqref{eq:Ev},
\eqref{eq:TaylorN}, \eqref{eq:TaylorR}, and
\eqref{eq:TaylorAverage} give
\begin{equation}\label{eq:recurrence0}
 \eps_{v+1}
 \leq
 C A^{-\kappa_0}\sigma_v^{-(d+3)}\eps_v
 +C\sigma_v^{-(d+3)}\Gamma_A\eps_v^2
 +C K^d e^{-K\sigma_v}\eps_v.
\end{equation}
Put
\begin{equation}\label{eq:Bv}
 B_v=C\sigma_v^{-(d+3)}K^{5d+8}.
\end{equation}
Since $A^{-q}\leq1$, \eqref{eq:recurrence0} implies
\begin{equation}\label{eq:recurrence}
 \eps_{v+1}
 \leq C_vA^{-\kappa_0}\eps_v
 +B_v\eps_v^2
 +CK^dA^{-c_*\sigma_v}\eps_v.
\end{equation}

Choose $c_*$ so that
\begin{equation}\label{eq:cstarchoice}
 c_*\min_{0\leq v<M}\sigma_v\geq L+\sigma_0+M\kappa_0+10.
\end{equation}
For $A$ sufficiently large, the second term in
\eqref{eq:recurrence} is bounded by $A^{-L-5}$.
An induction now gives
\begin{equation}\label{eq:epsinduction}
 \eps_v\leq A^{-\sigma_0-\frac12v\kappa_0+1},
 \qquad 0\leq v\leq M.
\end{equation}
Indeed, if \eqref{eq:epsinduction} holds at $v$, then
\[
 C_vA^{-\kappa_0}\eps_v+B_v\eps_v^2
 \leq
 C_v(\log A)^{C_v}
 A^{-\sigma_0-(v+1)\kappa_0+2}
 \leq
 \frac12A^{-\sigma_0-\frac12(v+1)\kappa_0+1},
\]
and \eqref{eq:cstarchoice} gives the same bound for the Fourier tail.
The estimate \eqref{eq:Westimate} and \eqref{eq:epsinduction} also verify
the smallness condition on $\eta_v$, closing the induction.

\smallskip
\noindent\emph{Step 5: termination.}
By \eqref{eq:Mchoice} and \eqref{eq:epsinduction},
\[
 \eps_M\leq
 A^{-\sigma_0-\frac12M\kappa_0+1}
 \leq A^{-L-3}\leq A^{-L}.
\]
Moreover, summing \eqref{eq:widths},
\[
 r_M\geq r_0-\sum_{v=0}^{\infty}\frac{r_0}{2^{v+3}}
 \geq\frac34r_0,
\qquad
 s_M\geq\frac{s_0}{4}
 -4\sum_{v=0}^{\infty}\frac{s_0}{2^{v+6}}
 \geq\frac{s_0}{8}.
\]
This proves \eqref{eq:finalsmall} and completes the finite induction.
\end{proof}

\section{Removal of the averaged quasi-periodic terms}

Write \eqref{eq:NFfinal} as
\begin{equation}\label{eq:averageSystem}
 \dot r=a(r,\theta)+F(r,\psi,\theta),\qquad
 \dot\psi=\nu(r)+b(r,\theta)+G(r,\psi,\theta),\qquad
 \dot\theta=\omega,
\end{equation}
where $\norm{(F,G)}\leq A^{-L}$.
By oddness of $a$,
\begin{equation}\label{eq:amean}
 [a](r):=\frac1{(2\pi)^d}\int_{\T^d}a(r,\theta)\,d\theta=0.
\end{equation}
Solve
\begin{equation}\label{eq:Qeq}
 \omega\cdot\partial_\theta Q(r,\theta)=a(r,\theta),
 \qquad [Q](r)=0.
\end{equation}
If
$a=\sum_{\ell\neq0}\widehat a(r,\ell)e^{\ii\langle\ell,\theta\rangle}$,
then
\begin{equation}\label{eq:Qformula}
 Q(r,\theta)=
 \sum_{\ell\neq0}
 \frac{\widehat a(r,\ell)}
 {\ii\langle\ell,\omega\rangle}
 e^{\ii\langle\ell,\theta\rangle}.
\end{equation}
From \eqref{eq:DC}, for $0<\sigma<s$,
\begin{equation}\label{eq:Qestimate}
 |Q|_{r,s-\sigma}
 \leq C\gamma^{-1}\sigma^{-2d-2}|a|_{r,s}.
\end{equation}
Since $a(r,-\theta)=-a(r,\theta)$, formula
\eqref{eq:Qformula} shows that $Q(r,-\theta)=Q(r,\theta)$.

Set
\begin{equation}\label{eq:rQ}
 r=u+Q(u,\theta).
\end{equation}
Because $Q(u,-\theta)=Q(u,\theta)$, the map
\[
 (u,\psi,\theta)\longmapsto
 (u+Q(u,\theta),\psi,\theta)
\]
satisfies \eqref{eq:equivariance}.  Lemma \ref{lem:revchange} therefore
shows that the transformed system remains reversible.  In particular,
\begin{equation}\label{eq:tildeParity}
 \begin{aligned}
 \widetilde F(u,-\psi,-\theta)
 &=-\widetilde F(u,\psi,\theta),\\
 \mathcal N(u,-\theta)&=\mathcal N(u,\theta),\\
 \widetilde G(u,-\psi,-\theta)
 &=\widetilde G(u,\psi,\theta).
 \end{aligned}
\end{equation}
The exact radial equation is
\begin{equation}\label{eq:uexact}
 \dot u=
 \bigl(1+\partial_uQ(u,\theta)\bigr)^{-1}
 F\bigl(u+Q,\psi,\theta\bigr).
\end{equation}
Hence
\begin{equation}\label{eq:utilde}
 \dot u=\widetilde F(u,\psi,\theta),\qquad
 \norm{\widetilde F}\leq CA^{-L}.
\end{equation}

After \eqref{eq:rQ}, write the angular equation as
\begin{equation}\label{eq:angularBefore}
 \dot\psi=\mathcal N(u,\theta)+\widetilde G(u,\psi,\theta),
\qquad \norm{\widetilde G}\leq CA^{-L}.
\end{equation}
Let
\[
 \Omega(u)=[\mathcal N](u)
 =\frac1{(2\pi)^d}\int_{\T^d}\mathcal N(u,\theta)\,d\theta
\]
and solve
\begin{equation}\label{eq:Weq}
 \omega\cdot\partial_\theta W(u,\theta)
 =\mathcal N(u,\theta)-\Omega(u),\qquad [W](u)=0.
\end{equation}
Because $\mathcal N(u,-\theta)=\mathcal N(u,\theta)$,
$W(u,-\theta)=-W(u,\theta)$.  With
\begin{equation}\label{eq:chi}
 \chi=\psi-W(u,\theta),
\end{equation}
we obtain
\begin{equation}\label{eq:finalSystem}
 \dot u=P(u,\chi,\theta),\qquad
 \dot\chi=\Omega(u)+Q_1(u,\chi,\theta),\qquad
 \dot\theta=\omega,
\end{equation}
where
\begin{equation}\label{eq:PQsmall}
 \norm{(P,Q_1)}\leq CA^{-L}.
\end{equation}
In the direction from the new variables to the old ones,
\[
 \psi=\chi+W(u,\theta).
\]
Since $W(u,-\theta)=-W(u,\theta)$,
\[
 -\chi+W(u,-\theta)
 =-\bigl(\chi+W(u,\theta)\bigr).
\]
Thus the angular shear also satisfies \eqref{eq:equivariance}.
Lemma \ref{lem:revchange} proves that \eqref{eq:finalSystem} is reversible,
and hence
\begin{equation}\label{eq:finalPQparity}
 P(u,-\chi,-\theta)=-P(u,\chi,\theta),\qquad
 Q_1(u,-\chi,-\theta)=Q_1(u,\chi,\theta).
\end{equation}
Indeed, using \eqref{eq:utilde}, \eqref{eq:angularBefore} and
\eqref{eq:Weq},
\begin{align}
 P&=\widetilde F,\label{eq:Pexact}\\
 Q_1&=\widetilde G
 -\partial_uW(u,\theta)\widetilde F.
\label{eq:Q1exact}
\end{align}
Furthermore, from \eqref{eq:nu},
\begin{equation}\label{eq:twist}
 \Omega(u)=c_0u^q+O(u^{q-\alpha}),\qquad
 \Omega'(u)=c_0qu^{q-1}
 +O(u^{q-\alpha-1}),
\end{equation}
and hence
\begin{equation}\label{eq:twistbound}
 cA^{-2/(n+2)}
 \leq|\Omega'(u)|
 \leq CA^{-2/(n+2)}.
\end{equation}

\section{Reversible KAM tori and boundedness}

We use the following standard reversible KAM theorem; see
\cite{Moser1967,Sevryuk,Liu2005,LiQiYuan}.

\begin{theorem}[Reversible KAM theorem]\label{thm:KAM}
Let
\[
 \dot u=P(u,\chi,\theta),\qquad
 \dot\chi=\Omega(u)+Q(u,\chi,\theta),\qquad
 \dot\theta=\omega
\]
be analytic and reversible under
$(u,\chi,\theta)\mapsto(u,-\chi,-\theta)$.
Assume \eqref{eq:DC},
\[
 \inf_{u\in I}|\Omega'(u)|\geq\tau>0,
\qquad \norm{(P,Q)}\leq\eps.
\]
There are constants $C,a>0$ such that, if
\begin{equation}\label{eq:KAMsmallness}
 \eps\leq C\tau^a,
\end{equation}
then a positive-measure Cantor subset of $I$ parameterizes invariant
analytic graphs
\[
 u=u_*(\chi,\theta)
\]
carrying quasi-periodic flows with frequencies
$(\Omega_*,\omega)$.
\end{theorem}

Take
\begin{equation}\label{eq:Lchoice}
 L>\frac{2a}{n+2}+2.
\end{equation}
By \eqref{eq:PQsmall} and \eqref{eq:twistbound},
\[
 \eps\leq CA^{-L}
 =o\left(A^{-2a/(n+2)}\right)
 =o(\tau^a),
\]
so \eqref{eq:KAMsmallness} holds for all sufficiently large $A$.
Theorem \ref{thm:KAM} therefore gives a positive-measure family of
invariant $(d+1)$-tori in every sufficiently large annulus selected in
Lemma \ref{lem:nonresonance}.

\begin{proof}[Proof of Theorem \ref{thm:main}]
Let $\mathcal T_A$ be one of the KAM tori obtained above.  In the final
coordinates it is a graph
\[
 \mathcal T_A=
 \{(u,\chi,\theta):u=u_A(\chi,\theta)\}.
\]
The transformations \eqref{eq:AA}, $\Psi^{(M)}$,
\eqref{eq:rQ}, and \eqref{eq:chi} are diffeomorphisms on their domains.
Thus their image $\widehat{\mathcal T}_A$ is an embedded invariant
codimension-one torus in
\[
 (0,\infty)\times\T^{d+1}.
\]
The complement has two components,
\begin{align*}
 \mathcal U_A^-&=
 \{(u,\chi,\theta):0<u<u_A(\chi,\theta)\},\\
 \mathcal U_A^+&=
 \{(u,\chi,\theta):u>u_A(\chi,\theta)\}.
\end{align*}
If an orbit crossed $\mathcal T_A$, uniqueness would give a first
intersection time and then force the orbit through that point to coincide
with the orbit contained in $\mathcal T_A$, a contradiction.  Hence
\begin{equation}\label{eq:barrier}
 z(0)\in\mathcal U_A^-\quad\Longrightarrow\quad
 z(t)\in\mathcal U_A^-\quad\text{for all }t
\end{equation}
in the maximal interval of existence.

Given initial data, choose $A$ so large that the initial action lies in
$\mathcal U_A^-$.  By \eqref{eq:barrier},
\[
 0<\rho(t)\leq C A .
\]
Using \eqref{eq:AA} and \eqref{eq:SCproperties},
\begin{equation}\label{eq:xybound}
 |x(t)|\leq C A^\alpha,\qquad
 |y(t)|\leq C A^\beta.
\end{equation}
The vector field in \eqref{eq:extended} is locally Lipschitz.
If a maximal solution had a finite endpoint $t_*$, then
\[
 \limsup_{t\to t_*}\bigl(|x(t)|+|y(t)|\bigr)=\infty,
\]
contradicting \eqref{eq:xybound}.  Thus the solution exists for all
$t\in\R$ and is bounded.
\end{proof}

\begin{remark}
For genuinely quasi-periodic forcing ($d>1$), there is no ordinary
time-one Poincar\'e map.  The barriers used above are codimension-one tori
of the extended flow.  In the same generality one should not claim
subharmonic solutions, because the forcing itself has no common period.
\end{remark}

\section*{Declarations}
\textbf{Conflict of interest.} The author declares no conflict of interest.

\textbf{Data availability.} No data were used in this study.

\end{document}